\documentclass{article}
\usepackage[latin1]{inputenc}
\usepackage{epsfig}
\usepackage{color}
\usepackage[british,english]{babel}
\usepackage{amsthm}
\usepackage{amsmath}
\usepackage{amsfonts}
\usepackage{amssymb}
\usepackage{graphicx}

\newtheorem{corollary}{Corollary}[section]
\newtheorem{definition}[corollary]{Definition}

\newtheorem{lemma}[corollary]{Lemma}
\newtheorem{prp}[corollary]{Proposition}

\newtheorem{thm}[corollary]{Theorem}

\newfont{\sBlackboard}{msbm10 scaled 900}

\newcommand{\mylabel}[1]{\label{#1}
            \ifx\undefined\stillediting
            \else \fbox{$#1$}\fi }
\newcommand{\BE}{\begin{equation}}

\newcommand{\EEQ}{\end{equation}}
\newcommand{\rfb}[1]{\mbox{\rm
   (\ref{#1})}\ifx\undefined\stillediting\else:\fbox{$#1$}\fi}

\newfont{\Blackboard}{msbm10 scaled 1200}

\newfont{\roma}{cmr10 scaled 1200}

\def\CC{\rm \hbox{C\kern-.56em\raise.4ex
         \hbox{$\scriptscriptstyle |$}\kern+0.5 em }}

\newcommand{\mm}    {{\hbox{\hskip 0.5pt}}}

\newcommand{\bluff} {{\hbox{\raise 15pt \hbox{\mm}}}}
scaled\magstep2

\makeatletter
\def\section{\@startsection {section}{1}{\z@}{-3.5ex plus -1ex minus
    -.2ex}{2.3ex plus .2ex}{\large\bf}}
\makeatother
\def\be{\begin{equation}}
\def\ee{\end{equation}}

\title{\bf Existence and nonexistence of solutions for p(x)-curl systems arising in electromagnetism}

\date{}

\begin{document}
\maketitle
\begin{center}
{\large \bf Anouar Bahrouni$\,^1$
 and
Du\v{s}an Repov\v{s}$\,^2$} 
\end{center}

\begin{center}
$^1\,$ Mathematics Department, University of
Monastir, Faculty of Sciences, \\
5019 Monastir, Tunisia, {\tt
bahrounianouar@yahoo.fr} 
\end{center}

\begin{center} 
$^2\,$
Faculty of Mathematics and Physics,
University of Ljubljana,  \\
1000
Ljubljana, Slovenia, {\tt
dusan.repovs@guest.arnes.si}
\end{center}

\noindent\textbf{Abstract.} {In this paper, we study the existence
and the nonexistence of solutions for a new class of $p(x)-$curl
systems arising in electromagnetism. This work generalizes some
results obtained in the $p-$curl case. There seems to be
no results on the nonexistence of solutions for curl systems with variable exponent.}
\section{Introduction}
The study of PDE's involving variable exponents has become very
attractive in recent decades since differential operators involving
variable exponent growth conditions can serve in describing
non-homogeneous phenomena which can occur in different branches of
science, e.g.: electrorheological fluids and
nonlinear Darcy's law in porous media, see \cite{a,aa}. \\
The literature on equations with $p(x)-$Laplacian or $p-$curl
operators is quite large, see e.g. \cite{Bartsh}-\cite{repovsaa} and
the references therein. To the best of our knowledge, the only
results involving the $p(x)-$curl operators can be found in
\cite{san,binlin}. In \cite{san}, the authors introduced a suitable
variable exponent Sobolev space and obtained the existence of local
or global weak solutions for equation with $p(x,t)-$curl operator by
using Galerkin's method. In \cite{binlin}, the authors used for the
first time the
variational methods for equations involving $p(x)-$curl operator.\\
 In this paper our aim is to study equations in which a variable
exponent curl operator is present. More precisely, we study the
existence and nonexistence of solutions. To the best of the authors
knowledge, this is one of the first works devoted to the studies of
the nonexistence of solutions
in variable exponent curl operator.\\
 Let $\Omega$ be a bounded simply connected domain of $\mathbb{R}^{3}$ with a $C^{1,1}$ boundary denoted by $\partial \Omega$.
 To introduce our problem more precisely, we first give some notations. Let
$u=(u_{1},u_{2},u_{3})$ be a vector function on $\Omega$. The
divergence of $u$ is denoted by
$$\nabla.u=\partial_{x_{1}}u_{1}+\partial_{x_{2}}u_{2}+\partial_{x_{3}}u_{3}$$
and the curl of $u$ is defined by
$$\nabla \times u=(\partial_{x_{2}}u_{3}-\partial_{x_{3}}u_{2},\partial_{x_{3}}u_{1}-\partial_{x_{1}}u_{3}, \partial_{x_{1}}u_{2}-\partial_{x_{2}}u_{1}).$$
We consider the following p(x)-curl systems:
\begin{equation}\label{eq}
\begin{cases}
\nabla\times (|\nabla\times u|^{p(x)-2}\nabla\times u)=\lambda
g(x,u)-\mu f(x,u), \ \ \ \nabla.u=0 \ \ \mbox {in} \ \ \Omega,\\
|\nabla\times u|^{p(x)-2}\nabla\times u\times \textbf{n}=0, \ \ \ \
\ \ \ \ \ \ \ \ \  \ \ \ \ \ \ \ \ \ \ \ \  \ \ \ \ \ u.\textbf{n}=0
\ \ \mbox{on} \ \
\partial \Omega,
\end{cases}
\end{equation}
where $\lambda,\mu>0$, $p\in C(\overline{\Omega})$ with
$\frac{6}{5}<p^{-}=\displaystyle \min_{x\in
\overline{\Omega}}p(x)\leq p^{+}=\displaystyle \max_{x\in
\overline{\Omega}}p(x)<3$ and there exists
$w:\mathbb{R}_{0}^{+}\rightarrow \mathbb{R}_{0}^{+}$ such that
$$\forall x,y \in \overline{\Omega}, \ \ |x-y|<1, \ \ |p(x)-p(y)|\leq w(|x-y|), \ \ \mbox{and} \ \ \displaystyle \lim_{s \rightarrow 0^{+}}w(s)log(\frac{1}{s})=C<\infty \ \ \  \ \ \ \ (P).$$
Throughout this paper we shall always make the following assumptions:\\
$(F_{1})$ $F: \Omega \times \mathbb{R}^{3} \rightarrow \mathbb{R}$
is differentiable with respect to $u\in \mathbb{R}^{3}$ and
$f=\partial_{u} F(x,u):\Omega \times \mathbb{R}^{3} \rightarrow
\mathbb{R}^{3}$ is a Carath\'eodory function.\\
$(F_{2})$ There exist $\alpha,\beta>0$  and $q\in
C(\overline{\Omega)}$ such that $
p^{+}<q(x)<p^{\ast}(x)=\frac{3p(x)}{3-p(x)}$ in $\overline{\Omega}$
and
$$\displaystyle F(x,u)\geq \alpha |u|^{q(x)} \ \ \mbox{and} \ \  |f(x,u)|\leq \beta(1+|u|^{q(x)-1}), \ \ \forall
(x,u)\in \overline{\Omega} \times \mathbb{R}^{3}.$$ $(G_{1})$ There
exist a nonnegative function $g\in L^{\infty}(\Omega)$ and $r\in
C(\overline{\Omega})$ such that
$$ p^{+}< r^{-}\leq r(x)<q^{-} \ \ and \ \ G(x,u)= g(x)|u|^{r(x)},$$
for all $(x,u)\in \overline{\Omega} \times \mathbb{R}^{3}.$\\
 $(G_{2})$ $G: \Omega \times \mathbb{R}^{3}
\rightarrow \mathbb{R}$ is differentiable with respect to $u\in
\mathbb{R}^{3}$ and $g=\partial_{u} G(x,u):\Omega \times
\mathbb{R}^{3} \rightarrow
\mathbb{R}^{3}$  is a Carath\'eodory function.\\
$(G_{3})$ There exist $\gamma,\mu>0$, $L>1$ and $k,r\in
C(\overline{\Omega})$ such that $1< k<p^{-}$ and\\
 $1<
r(x)<p^{\ast}(x)$,
$$|g(x,u)|\leq \mu(1+|u|^{r(x)-1}), \forall (x,u)\in\overline{\Omega}\times
\mathbb{R}^{3},
 \ \ \displaystyle \limsup_{u\rightarrow 0} \frac{G(x,u)}{|u|^{p^{+}}}=0 \ \ \mbox{uniformly in } \ \ x\in \Omega$$
and
 $$\displaystyle \sup_{u\in E}\displaystyle \int_{\Omega}G(x,u)dx>0, \ \ |G(x,u)|\leq \gamma |u|^{k(x)}, \forall x\in \mathbb{R}^{3}, \forall |u|>L.$$
 Our main results are the following two theorems.
\begin{thm}\label{thm1}
Assume that hypotheses $(F_{1})$-$(F_{2})$ and $(G_{1})-(G_{2})$
hold. Then:\\ $(i)$ There exist $\lambda_{1},\mu_{1}>0$ such that,
if $0<\lambda<\lambda_{1}$ and $\mu>\mu_{1}$, then problem
\eqref{eq} does not have any nontrivial weak solution.\\
$(ii)$ For each $\mu>0$, there exists $\lambda_{\mu}>0$ such that if
$\lambda>\lambda_{\mu}$, then problem \eqref{eq} has at least one
nontrivial weak solution.
\end{thm}
\begin{thm}\label{thm2}
Assume that $(F_{1})-(F_{2})$ and $(G_{2})-(G_{3})$ hold. Then there
exist $\lambda_{2},\lambda_{3},r>0$ such that, if $\lambda \in
[\lambda_{2},\lambda_{3}]$, there exists $ \mu_{2}>0$ with the
following property: for each $\mu\in [0,\mu_{2}]$, equation
\eqref{eq} has at least three solutions whose norms are less than
$r$.
\end{thm}
We have divided this paper into $3$ sections. In section $2$ we
give some notations and we recall some necessary definitions. In
section $3$ we prove our main results.
\section{Function spaces with variable exponent and preliminary results}
In this section we recall some basic definitions and properties
concerning the basic function spaces with variable exponent and the
space $\textbf{W}^{p(x)}(\Omega)$ of divergence free vector
functions belonging to $\textbf{L}^{p(x)}(\Omega)$ with curl in
$\textbf{L}^{p(x)}(\Omega)$. We refer to \cite{san,Fan,
Fann,zho,musi,radrep,binlin} and the references therein.\\
Consider the set
$$C_+(\overline\Omega)=\{p\in C(\overline\Omega),\;p(x)>1\;{\rm
for}\; {\rm all}\;x\in\overline\Omega\}.$$ For any $p\in
C_+(\overline\Omega)$ define
    $$p^+=\sup_{x\in\overline{\Omega}}p(x)\qquad\mbox{and}\qquad p^-=
    \inf_{x\in\overline{\Omega}}p(x),$$ and the {\it variable exponent Lebesgue space}
    $$L^{p(x)}(\Omega)=\left\{u;\ u\ \mbox{is
measurable real-valued function such that }
 \int_\Omega|u(x)|^{p(x)}\;dx<\infty\right\}.$$ This vector space is
a Banach space if it is endowed with the {\it Luxemburg norm}, which
 is defined by
 $$|u|_{p(x)}=\inf\left\{\mu>0;\;\int_\Omega\left|
 \frac{u(x)}{\mu}\right|^{p(x)}\;dx\leq 1\right\}.$$ The function
 space $L^{p(x)}(\Omega)$ is reflexive if and only if $1 < p^-\leq
 p^+<\infty$
  and continuous functions with compact support
 are dense in $L^{p(x)}(\Omega)$ if $p^{+}<\infty$.

 Let $L^{q(x)}(\Omega)$ denote the conjugate space of
 $L^{p(x)}(\Omega)$, where $1/p(x)+1/q(x)=1$. If $u\in
 L^{p(x)}(\Omega)$ and $v\in L^{q(x)}(\Omega)$ then  the following
 H\"older-type inequality holds:
 \begin{equation}\label{Hol}
 \left|\int_\Omega uv\;dx\right|\leq\left(\frac{1}{p^-}+
 \frac{1}{q^-}\right)|u|_{p(x)}|v|_{q(x)}\,.
 \end{equation}
 Moreover, if $p_j\in C_+(\overline\Omega)$ ($j=1,2,3$) and
 $$\frac{1}{p_1(x)}+\frac{1}{p_2(x)}+\frac{1}{p_3(x)}=1$$
 then for all $u\in L^{p_1(x)}(\Omega)$, $v\in L^{p_2(x)}(\Omega)$,
 $w\in L^{p_3(x)}(\Omega)$
 \begin{equation}\label{Hol1}
 \left|\int_\Omega uvw\;dx\right|\leq\left(\frac{1}{p_1^-}+
 \frac{1}{p_2^-}+\frac{1}{p_3^-}\right)|u|_{p_1(x)}|v|_{p_2(x)}|w|_{p_3(x)}\,.
 \end{equation}

 The inclusion between Lebesgue spaces also generalizes the classical
 framework, namely if $0 <|\Omega|<\infty$ and $p_1$, $p_2$ are
 variable exponents so that $p_1\leq p_2$
  in $\Omega$, then there exists the continuous embedding
 $L^{p_2(x)}(\Omega)\hookrightarrow L^{p_1(x)}(\Omega)$.
 \begin{prp}\label{pr1}
 If we denote
 $$\rho(u)=\displaystyle \int_{\Omega}|u|^{p(x)}dx, \ \ \forall u\in L^{p(x)}(\Omega),$$
 then\\
 $(i)$ $|u|_{p(x)}<1 (\mbox{resp}., =1;>1)\Leftrightarrow \rho(u)<1(\mbox{resp}., =1;>1)$;\\
 $(ii)$ $|u|_{p(x)}>1 \Rightarrow |u|_{p(x)}^{p^{-}}\leq \rho(u) \leq
 |u|_{p(x)}^{p^{+}}$; and\\
 $(iii)$ $|u|_{p(x)}<1 \Rightarrow |u|_{p(x)}^{p^{+}}\leq \rho(u)
 \leq |u|_{p(x)}^{p^{-}}$.
 \end{prp}
 \begin{prp}\label{pr2}
 If $u,u_{n}\in L^{p(x)}(\Omega)$ and  $n\in \mathbb{N}$, then the
 following statements are equivalent:\\
 $(1)$ $\displaystyle \lim_{n\rightarrow +\infty}
 |u_{n}-u|_{p(x)}=0$.\\
 $(2)$  $\displaystyle \lim_{n\rightarrow +\infty} \rho(
 u_{n}-u)=0.$\\
 $(3)$ $u_{n}\rightarrow u$ in measure in $\Omega$ and
 $\displaystyle \lim_{n\rightarrow +\infty}\rho(u_{n})=\rho(u).$
 \end{prp}
  If $k$ is a positive integer and $p\in C_+(\overline\Omega)$, we define the variable exponent  Sobolev space by
  $$
  W^{k,p(x)}(\Omega)=\{u\in L^{p(x)}(\Omega):\; D^\alpha u\in
  L^{p(x)} (\Omega),\ \mbox{for all}\ |\alpha|\leq k \}.
  $$
  Here, $\alpha =(\alpha_1,\ldots ,\alpha_N)$ is a multi-index, $|\alpha|=\sum_{i=1}^N\alpha_i$ and
  $$D^\alpha u=\frac{\partial^{|\alpha|}u}{\partial^{\alpha_1}_{x_1}\ldots \partial^{\alpha_N}_{x_N}}\,.$$

  On $W^{k,p(x)}(\Omega)$ we  consider the following norm
  $$
  \|u\|_{k,p(x)}=\sum_{|\alpha|\leq k}|D^\alpha u|_{p(x)}.
  $$
  Then $W^{k,p(x)}(\Omega)$ is a reflexive and separable Banach space. Let $W^{k,p(x)}_0(\Omega)$
  denote the closure of $C^\infty_0(\Omega)$ in
  $W^{k,p(x)}(\Omega)$.
\begin{thm}\label{rad}
Let $q\in C(\overline{\Omega})$ such that $1\leq
q(x)<\frac{3p(x)}{3-p(x)}$ in  $\overline{\Omega}$. Then the
embedding $W^{1,p(x)}(\Omega)\hookrightarrow
  L^{q(x)}(\Omega)$ is compact.
\end{thm}

  Let $\textbf{L}^{p(x)}(\Omega)=L^{p(x)}(\Omega)\times L^{p(x)}(\Omega) \times L^{p(x)}(\Omega)
  $ and define
  $$E=\textbf{W}^{p(x)}(\Omega)=\{v\in \textbf{L}^{p(x)}(\Omega): \nabla \times v \in \textbf{L}^{p(x)}(\Omega), \nabla .v=0, v.\textbf{n}|\partial \Omega=0 \},$$
  where $\textbf{n}$ denotes the outward unit normal vector to $\partial
  \Omega$. Equip $\textbf{W}^{p(x)}(\Omega)$ with the following norm
  $$\|v\|=\|v\|_{\textbf{L}^{p(x)}(\Omega)}+\|\nabla \times v\|_{\textbf{L}^{p(x)}(\Omega)}.$$
  If $p^{-}>1$, then by Theorem $2.1$ of \cite{san},
  $E=\textbf{W}^{p(x)}(\Omega)$ is a closed subspace of
  $\textbf{W}^{1,p(x)}_{\textbf{n}}(\Omega)$, where
  $$\textbf{W}^{1,p(x)}_{\textbf{n}}(\Omega)=\{v\in \textbf{W}^{1,p(x)}(\Omega),  v.\textbf{n}|\partial \Omega=0 \}$$
  and
  $$\textbf{W}^{1,p(x)}(\Omega)=W^{1,p(x)}(\Omega)\times W^{1,p(x)}(\Omega)\times W^{1,p(x)}(\Omega).$$
  Thus we have the following theorem.
  \begin{thm}\label{sann}
  Assume that $1<p^{-}\leq p^{+}<\infty$ and $p$ satisfies $(P)$. Then
  $\textbf{W}^{p(x)}(\Omega)$ is a closed subspace of
  $\textbf{W}^{1,p(x)}_{\textbf{n}}(\Omega)$. Moreover, if
  $p^{-}>\frac{6}{5}$, then $\|\nabla \times .\|$ is a norm in
  $\textbf{W}^{p(x)}(\Omega)$ and there exists
  $C=C(N,p^{-},p^{+})>0$ such that
  $$\|v\|_{\textbf{W}^{1,p(x)}(\Omega)}\leq C \|\nabla \times v\|_{{\textbf{L}}^{p(x)}(\Omega)}.$$
  \end{thm}
  \begin{corollary}\label{an}
  By Theorems \ref{rad} and \ref{sann}, the embedding $\textbf{W}^{p(x)}(\Omega)\hookrightarrow
  \textbf{L}^{q(x)}(\Omega)$ is compact, with $1<p^{-}\leq p^{+}<3, \ \ q\in
  C(\overline{\Omega})$ and $1\leq q(x)<\frac{3p(x)}{3-p(x)}$ in
  $\overline{\Omega}$. Moreover, $(\textbf{W}^{p(x)}(\Omega),
  \|.\|)$
  is a uniformly convex and reflexive Banach space.
  \end{corollary}
Define for any $\lambda,\mu>0$ and $u\in E$ ,
$$\phi(u)=\displaystyle \int_{\Omega}|\nabla \times
u|^{p(x)}dx, J(u)=\displaystyle \int_{\Omega}G(x,u)dx$$
$$\psi(u)=\displaystyle \int_{\Omega}-F(x,u)dx \ \ \mbox{and} \ \ I(u)=\phi(u)-\lambda J(u)-\mu \psi(u).$$
It is easy to see, under assumptions $(P)$, $(F_{1})-(F_{2})$ and
$(G_{1})-(G_{3})$, that $I,\phi,J,\psi\in C^{1}(E,\mathbb{R}).$
\begin{definition}
For every $\lambda,\mu>0$, we say that $u\in E$ is a weak solution
of problem \eqref{eq}, if
$$\displaystyle \int_{\Omega}|\nabla \times u|^{p(x)-2}\nabla \times u.\nabla \times v dx- \lambda \displaystyle \int_{\Omega}g(x,u(x)).vdx
+\mu\displaystyle \int_{\Omega}f(x,u(x)).v dx=0, \ \ \forall v\in
E.$$ For more details, we refer the reader to \cite{binlin}.
\end{definition}
\section{Proofs of main results}
\subsection{Proof of Theorem \ref{thm1}}
\begin{lemma}\label{non}
Suppose that the assumptions $(F_{1})-(F_{2})$ and $(G_{1})-(G_{2})$
are fulfilled. Then there exist positive constants
$\lambda_{1},\mu_{1}$ such that, for every $0<\lambda<\lambda_{1}$
and $ \mu_{1}<\mu$, problem \eqref{eq} does not have any nontrivial
weak solutions.
\end{lemma}
\begin{proof}
Assume that $u$ is a nontrivial weak solution of equation \eqref{eq}. \\
\textbf{Case} $1$: We suppose that $\|u\|<1$. Then, by
  Proposition \ref{pr1}, $u$ satisfies the following
inequality
\begin{equation}\label{1}
\|u\|^{p^{+}}\leq\displaystyle \int_{\Omega}|\nabla \times
u|^{p(x)}dx=\lambda \displaystyle \int_{\Omega}G(x,u)dx-\mu
\displaystyle \int_{\Omega}F(x,u)dx.
\end{equation}
 Now, since
$r(x)<q(x)$ in $\Omega$, applying the Young inequality we can deduce
that
\begin{equation}\label{2}
\lambda \displaystyle \int_{\Omega}g(x)|u|^{r(x)}dx\leq
\frac{q^{+}-r^{-}}{q^{-}}\displaystyle \int_{\Omega} |\lambda
g|^{\frac{q(x)}{q(x)-r(x)}}dx +\frac{r^{+}}{q^{-}}\displaystyle
\int_{\Omega} |u|^{q(x)}dx.
\end{equation}
Invoking inequalities \eqref{1} and \eqref{2}, and conditions
$(F_{1})-(F_{2})$ and $(G_{1})$, for $\lambda$ small enough, we
obtain
\begin{align}\label{3}
0<\|u\|^{p^{+}}
&\leq\frac{(q^{+}-r^{-})\lambda^{\frac{q^{-}}{q^{+}-r^{-}}}}{q^{-}}\displaystyle
\int_{\Omega} |g|^{\frac{q(x)}{q(x)-r(x)}}dx+
(\frac{r^{+}}{q^{-}}-\mu \alpha)\displaystyle
\int_{\Omega} |u|^{q(x)}dx \nonumber\\
&\leq\frac{(q^{+}-r^{-})\lambda^{\frac{q^{-}}{q^{+}-r^{-}}}}{q^{-}}\displaystyle
\int_{\Omega} |g|^{\frac{q(x)}{q(x)-r(x)}}dx=
\lambda^{\frac{q^{-}}{q^{+}-r^{-}}} A <\infty,
\end{align}
where $A=\frac{q^{+}-r^{-}}{q^{-}}\displaystyle \int_{\Omega}
|g|^{\frac{q(x)}{q(x)-r(x)}}dx$ and $\mu>\mu_{1}=\frac{r^{+}}{\alpha
q^{-}}$. \\ Thanks to Corollary \ref{an}, there exists a constant
$\beta>0$ such that
\begin{equation}\label{4}
\beta |u|_{r(x)}^{p^{+}}\leq \|u\|^{p^{+}}, \ \ \forall u\in E.
\end{equation}
Thus, in view of \eqref{1}, \eqref{4}, and Proposition \ref{pr1}, we
get
\begin{equation}\label{5}
\beta |u|_{r(x)}^{p^{+}}\leq \lambda \|g\|_{\infty}
\max(|u|_{r(x)}^{r^{-}},|u|_{r(x)}^{r^{+}}).
\end{equation}
 Having in mind $p^{+}<r^{-}<r^{+}$ and $\|u\|_{r(x)}>0$, by \eqref{3} and \eqref{5}, we
 have
\begin{equation}\label{6}
\beta \max((\frac{\beta}{\lambda  \|g\|_{\infty}
})^{\frac{p^{+}}{r^{-}-p^{+}}},(\frac{\beta}{\lambda  \|g\|_{\infty}
})^{\frac{p^{+}}{r^{+}-p^{+}}})\leq \|u\|^{p^{+}}\leq
\lambda^{\frac{q^{-}}{q^{+}-r^{-}}}A.
\end{equation}
\textbf{Case} $2$: We suppose that $\|u\|>1$. It suffices to replace $p^{+}$ by $p^{-}$ in the proof of Case $1$.\\
This concludes the proof.
\end{proof}
\begin{lemma}\label{coer}
Assume that assumptions $(F_{1})-(F_{2})$ and $(G_{1})-(G_{2})$ hold. Then\\
$a)$ $I$ is a coercive functional; and\\
$b)$ $I$ is a weakly lower semicontinuous functional.
\end{lemma}
\begin{proof}
$a)$ Let $\lambda,\mu>0$ and $u\in E$ with $\|u\|>1$. Combining
Proposition \ref{pr1}, Young inequality, and assumptions $(F_{2})$
and $(G_{1})$, one obtains the following inequalities
\begin{align*}
I(u)&\geq \frac{1}{p^{+}}\|u\|^{p^{-}}+\alpha \mu\displaystyle
\int_{\Omega} |u(x)|^{q(x)}dx- \lambda
\displaystyle \int_{\Omega} g(x)|u|^{r(x)}dx\\
&\geq \frac{1}{p^{+}}\|u\|^{p^{-}}+\frac{\mu\alpha}{2}\displaystyle
\int_{\Omega} |u(x)|^{q(x)}dx- c_{\lambda,\mu,\alpha}\displaystyle
\int_{\Omega}|g(x)|^{\frac{q(x)}{q(x)-r(x)}}dx,
\end{align*}
where $c_{\lambda,\mu,\alpha}$ is a positive constant. This demonstrates the coercivity of the functional $I$.\\
$b)$ Let $(u_{n})$ be a sequence such that $u_{n}\rightharpoonup u$
in $E$. Using the fact that $(u_{n})$ is bounded in $E$, Corollary
\ref{an} and Proposition \ref{pr2}, up to a subsequence, still
denoted by $(u_{n})$, we can infer that
\begin{equation}\label{7}
u_{n}\rightarrow u \ \ \mbox{a.e in} \ \ \Omega \ \ \mbox{and} \ \
\displaystyle \lim_{n\rightarrow
+\infty}\int_{\Omega}g(x)|u_{n}|^{r(x)}dx=\displaystyle
\int_{\Omega}g(x)|u|^{r(x)}dx.\\
\end{equation}
By the weak lower semicontinuity of the norm $\|.\|$, we have
\begin{equation}\label{8}
\|u\|\leq \displaystyle \liminf_{n\rightarrow +\infty} \|u_{n}\|.
\end{equation}
Furthermore, Fatou's lemma and $(F_{2})$ yield the following
inequality
\begin{equation}\label{9}
\displaystyle \int_{\Omega} \displaystyle \liminf_{n\rightarrow
+\infty} F(x,u_{n})dx \leq \liminf_{n\rightarrow +\infty}
\displaystyle \int_{\Omega} F(x,u_{n})dx.\\
\end{equation}
Combining \eqref{7}, \eqref{8} and \eqref{9}, we have thus proved
the claim.
\end{proof}
\noindent
\textbf{Completion of the proof of Theorem \ref{thm1}:}\\
$i)$ Evidently, by Lemma \ref{non}, assertion $(i)$ of Theorem \ref{thm1} holds.\\
$ii)$ Fix $\mu>0$. Using Lemma \ref{coer}, for every $\lambda>0$, we
can find $u\in E$ such that
$$I(u)=\displaystyle \inf_{v\in E}I(v).$$
Hence, for every $\lambda>0$ and $\mu>0$, $u$ is a weak solution of
problem \eqref{eq}. It remains to show that $u$ is nontrivial weak
solution of system \eqref{eq}. Invoking assumption $(G_{1})$, we can
find $w\in E$ such that
$$\displaystyle \int_{\Omega}G(x,w)dx=1.$$
It follows that
$$I(w)=\displaystyle \int_{\Omega} \frac{|\nabla \times w|^{p(x)}}{p(x)}dx+\mu\displaystyle \int_{\Omega}F(x,w)dx-\lambda=\lambda_{\mu}-\lambda,$$
where $\lambda_{\mu}=\displaystyle \int_{\Omega} \frac{|\nabla
\times w|^{p(x)}}{p(x)}dx+\mu\displaystyle \int_{\Omega}F(x,w)dx.$ Thus, $I(w)<0$ for any $\lambda>\lambda_{\mu}$.\\
This completes the proof.
\subsection{Proof of Theorem \ref{thm2}}
The main tool in the proof of Theorem \ref{thm2} is the variant of
the three critical points theorem established by Ricceri
\cite{ricceri}.
 Before stating his theorem, we need the following definition.
\begin{definition}
If $X$ is a real Banach space, we denote by $R_{X}$ the class of all
functional $\phi: X\rightarrow \mathbb{R}$ possessing the following
property: If $(u_{n})$ is a sequence in $X$, converging weakly to
$u\in X$, and $\displaystyle \liminf_{n\rightarrow
+\infty}\phi(u_{n})\leq \phi(u)$, then $(u_{n})$ has a subsequence
strongly converging  to $u$.
\end{definition}
\begin{thm}\label{ricceri}
Let $X$ be a separable and reflexive real Banach space,
$\phi:X\rightarrow \mathbb{R}$ a coercive, sequentially weakly lower
semicontinuous $C^{1}$ functional belonging to $R_{X}$, bounded on
each bounded subset of $X$ and with the derivative admitting a
continuous inverse on $X^{\ast}$; $J: X\rightarrow \mathbb{R}$, a
$C^{1}$ functional with compact derivative. Assume that $\phi$ has a
strict local minimum $x_{0}$ with $\phi(x_{0})=J(x_{0})=0$. Finally,
setting
$$\alpha=\max\{0,\displaystyle \limsup_{\|x\|\rightarrow +\infty}\frac{J(x)}{\phi(x)}, \displaystyle \lim_{\|x\|\rightarrow x_{0}}\frac{J(x)}{\phi(x)}\}$$
and
$$\beta=\displaystyle \sup_{x\in \phi^{-1}(0,+\infty)}\frac{J(x)}{\phi(x)},$$
assume that $\alpha<\beta.$ Then for each compact interval
$[a,b]\subset (\frac{1}{\beta},\frac{1}{\alpha})$ (with the
conventions $\frac{1}{0}=+\infty$, $\frac{1}{\infty}=0$), there
exists $r>0$ with the following property: for every $\lambda \in
[a,b]$ and every $C^{1}$ functional $\psi: X\rightarrow \mathbb{R}$
with compact derivative, there exists $\delta>0$ such that for each
$\mu \in [0,\delta]$, the equation
$$\phi^{'}(x)=\lambda J^{'}(x)+\mu \psi^{'}(x)$$
has at least three solutions whose norm is less than $r$.
\end{thm}
\noindent
\textbf{Completion of the proof of Theorem \ref{thm2}:} \\
Standard arguments can be used to show that $J^{'}$ and $\psi^{'}$
are compact, while $\phi$ is a coercive, sequentially weakly lower
semicontinuous and $\phi^{'}$ is  a homeomorphism between $E$ and
its dual. Clearly, $\phi\in R_{E}$, since $E$ is uniformly convex.\\
Fix $\epsilon,s>0$ such that $p^{+}+s<p^{\ast}(x)$. By virtue of
assumption $(G_{3})$, there exists a constant $\eta$ with
$0<\eta<L$, such that
$$G(x,u)\leq \epsilon |u|^{p^{+}}, \ \ \forall x\in \Omega, \forall
|u| \in [-\eta,\eta].$$ Again, by assumption $(G_{3})$, it follows
that
\begin{align*}
 J(u)&\leq  \displaystyle \int_{\{x\in \Omega,
|u(x)|\leq \eta\}}G(x,u)dx+\displaystyle \int_{\{x\in \Omega, \eta
\leq |u(x)|\leq L\}}G(x,u)dx+\displaystyle \int_{\{x\in \Omega,
|u(x)|\geq L\}}G(x,u)dx\\
&\leq c(\epsilon \displaystyle \int_{\{x\in \Omega, |u(x)|\leq
\eta\}}|u|^{p^{+}}dx+ \displaystyle \int_{\{x\in \Omega, \eta \leq
|u(x)|\leq L\}}|u|^{p^{+}+s}dx+ \gamma \displaystyle \int_{\{x\in
\Omega,
|u(x)|\geq L \}}|u|^{k(x)}dx)\\
&\leq c(\epsilon\displaystyle \int_{\{x\in \Omega, |u(x)|\leq
\eta\}}|u|^{p^{+}}dx+ \displaystyle \int_{\{x\in \Omega, \eta \leq
|u(x)|\leq L\}}|u|^{p^{+}+s}dx+ \displaystyle \int_{\{x\in \Omega,
|u(x)|\geq L\}}|u|^{p^{+}+s}dx)\\
&\leq c(\epsilon \|u\|^{p^{+}}+\|u\|^{p^{+}+s}),
 \end{align*}
 for some positive constant $c$.
This, along with Proposition \ref{pr1}, yields\\
for $\|u\|<1$
$$\frac{J(u)}{\phi(u)}\leq c \frac{\epsilon\|u\|^{p^{+}}+\|u\|^{p^{+}+s}}{\|u\|^{p^{+}}},$$
hence,
\begin{equation}\label{c1}
\displaystyle \limsup_{u\rightarrow 0} \frac{J(u)}{\phi(u)}\leq
c\epsilon.
\end{equation}
Taking $u\in E$ with $\|u\|>1$, from $(G_{2})-(G_{3})$, we get
\begin{align*}
J(u)&\leq\displaystyle \int_{\{x\in \Omega, |u(x)|\leq
L\}}G(x,u)dx+\displaystyle \int_{\{x\in \Omega, |u(x)|\geq L
\}}G(x,u)dx\\
&\leq c+ \displaystyle \int_{\{x\in \Omega, |u(x)|\geq L \}}
|u|^{k(x)}dx\\
&\leq c(1+\|u\|^{k^{-}}+\|u\|^{k^{+}}).
\end{align*}
This implies that
\begin{equation}\label{c2}
\displaystyle \limsup_{u\rightarrow +\infty}
\frac{J(u)}{\phi(u)}\leq \displaystyle \limsup_{u\rightarrow
+\infty} \frac{c(1+\|u\|^{k^{-}}+\|u\|^{k^{+}})}{\|u\|^{p^{-}}}=0.
\end{equation}
Therefore, by \eqref{c1} and \eqref{c2}, $\alpha=0$. In view of
assumption $(G_{3})$, we have $\beta>0$. Thus, all hypotheses of
Theorem \ref{ricceri} are satisfied. The proof is therefore
complete. \\
\medskip
{\bf Acknowledgement}. 
The second
author was supported by the Slovenian Research Agency grants
P1-0292, J1-7025, and J1-8131, and N1-0064.

\end{document}